\documentclass[preprint,number]{elsarticle}
\usepackage[latin1]{inputenc}
\usepackage[english]{babel}
\usepackage{amsfonts,amsmath,amsthm}
\newtheorem{thm}{Theorem}

\begin{document}

\title{Complexity reduction of C-Algorithm}
\author[mb]{Magali Bardet}
\ead{Magali.Bardet@univ-rouen.fr}
 \address[mb]{Laboratoire d'Informatique, de Traitement de l'Information et des Syst\`emes,\\
  Universit\'e de Rouen \\
  Avenue de l'Université BP 12\\
  76801 Saint \'Etienne du Rouvray \;FRANCE}

\author[ib]{Islam Boussaada\corref{cor2}}
\ead{islam.boussaada@gmail.com}
\address[ib]{Laboratoire de Mathématiques Rapha\"el Salem,\\
  CNRS, Universit\'e de Rouen \\
  Avenue de l'Université BP 12\\
  76801 Saint \'Etienne du Rouvray \;FRANCE\\
 and\\
  Laboratoire des signaux et syst\`emes (L2S)\\
 Sup\'elec - 3 rue Joliot-Curie\\
 91192 Gif-sur-Yvette cedex \;FRANCE}

\cortext[cor2]{Corresponding author}


\begin{abstract}
  The C-Algorithm introduced in~\cite{Chouikha2007} is designed to
  determine isochronous centers for Lienard-type differential systems,
  in the general real analytic case. However, it has a large
  complexity that prevents computations, even in the quartic
  polynomial case.

  The main result of this paper is an efficient algorithmic
  implementation of C-Algorithm, called ReCA (Reduced
  C-Algorithm). Moreover, an adapted version of it is proposed in the
  rational case. It is called RCA (Rational C-Algorithm) and is widely
  used in~\cite{BaBoChSt2010} and~\cite{BoussaadaChouikhaStrelcyn2010}
  to find many new examples of isochronous centers for the Li\'enard
  type equation.
\end{abstract}
\maketitle
\section{Introduction}
The use of symbolic computations and Computer Algebra Systems in the
Qualitative investigations of ordinary differential equations becomes
a standard way, see for instance \cite{DumortierLlibreArtes2006}. One
of the important problems in this field is the characterization and
the explicit description of isochronous centers for planar polynomials
vector fields. See for example~\cite{RomanovskiShafer2009}.

This is a companion paper
of~\cite{BaBoChSt2010,BoussaadaChouikhaStrelcyn2010} which are devoted
to the seeking out of isochronous centers for real planar polynomial
Li\'enard type equation. These investigations are based on the
C-Algorithm introduced in~\cite{Chouikha2007} and used
in~\cite{ChouikhaRomanovskiChen2007} for the cubic case.

The C-Algorithm in its original form has a large computational
cost. In the real analytic case as well as in the particular rational
case the careful inspection of the formulas used leads to more
efficient algorithms, called ReCA (Reduced C-Algorithm) and RCA
(Rational C-Algorithm) respectively. The aim of this note is to give a
detailed description of them.

Consider the Li\'enard type differential equation
\begin{equation}\label{L2} 
\ddot x+f(x){\dot x}^2+g(x)=0
\end{equation} 
where $f$ and $g$ are defined in a neighborhood of $0\in \mathbb{R}$, or equivalently its associated two dimensional (planar)
system
\begin{equation}\label{L2P}
  \left.\begin{aligned} \dot x &= y \\ \dot y &= -g(x) - f(x) y^2
    \end{aligned}\right\}. 
\end{equation}
In this paper, we assume that $g(0)=0$, which insures that $O$ is a
critical point, and $xg(x)>0$ in a punctured neighborhood of $0\in \mathbb{R}$,
which insure that the origin is a center. Moreover, we suppose that
$g'(0)=1$, so that system~\eqref{L2P} is a perturbation of the linear
isochronous center $\dot x = y, \dot y = -x$.

In our knowledge equation~\eqref{L2} appears for the first time in M.~Sabatini's paper~\cite{Sabatini2004}, where sufficient conditions for the isochronicity
of the origin $O$ are given for  $C^1$ functions $f$ and $g$ defined in some
neighborhood of~$0$.

In the real analytic case, necessary and sufficient conditions for
isochronicity are given by A.~R.~Chouikha in~\cite{Chouikha2007},
where Theorem~\ref{thmC} is proved. We use the same notations as in~\cite{Chouikha2007}:
\begin{equation}\label{eq:F}
F(x):= \int_0^x f(s) ds, \quad \phi(x):= \int_0^x e^{F(s)} ds.
\end{equation}

As $xg(x)>0$ for $x\neq0$, the function $\xi$ is well defined by the
relation
\begin{equation}\label{eq:xi}
\frac {1}{2} \xi(x)^2 = \int_0^x g(s) e^{2F(s)} ds
\end{equation}
and the condition $x\xi(x)>0,\;\forall x\neq 0$. Such function $\xi$ is real analytic in some neighborhood of $0$ and $\xi'(0)=1$.
\begin{thm}[Chouikha,\cite{Chouikha2007}]
\label{thmC}
Let $f$ and $g$ be real analytic functions in a neighborhood of $0$, such that
$g(0)=0$ and $x g(x) > 0$ for $x \neq 0$.  Then system~\eqref{L2} has
an isochronous center at $O$ if and only if there exists an odd real 
analytic function $h$, called the Urabe function, satisfying the
following conditions:
 \begin{equation}
 \label{CRI}
 \frac {\xi(x)}{1+h(\xi(x))} = g(x) e^{F(x)}, 
 \end{equation}
 and
 \begin{equation}\label{bb}
  \phi(x) = \xi(x) + \int_0^{\xi(x)} h(t) dt,
 \end{equation}
where $F(x), \phi(x)$ and $\xi(x)$ are defined by~\eqref{eq:F}--\eqref{eq:xi}.
\end{thm}
Taking into account~\eqref{eq:xi}, it is easy to see that~\eqref{CRI} and~\eqref{bb} are equivalent.

This theorem leads to an algorithmic method, named C-Algorithm, which
gives necessary conditions for isochronicity by computing the first
coefficients of the power series expansion around~$0$ of both sides
of~\eqref{CRI}. Then, the sufficiency is insured by establishing
explicitly the odd Urabe function. More details about the C-Algorithm
can be found
in~\cite{Chouikha2007,ChouikhaRomanovskiChen2007,BoussaadaChouikhaStrelcyn2010,BaBoChSt2010}.

In particular, in~\cite{BoussaadaChouikhaStrelcyn2010,BaBoChSt2010},
the authors investigate the practical applicability of the
C-Algorithm to the following family of planar polynomial systems,
which are perturbations of the linear isochronous center $\dot x = y,
\dot y = -x$:
\begin{equation}\label{CHERKAS}
  \left.\begin{aligned} \dot x &= p_{{0}} \left( x \right) +p_{{1}} \left( x \right) y \\ \dot y &= q_{{0}} \left( x \right) +q_{{1}} \left( x \right) y+q_{{2}} \left( x
      \right) {y}^{2}
    \end{aligned}\right\}, 
\end{equation} 
where $p_0,p_1,q_0,q_1,q_2 \in \mathbb{R}[x]$, $p_0(0)=p_0'(0)=0,\,
p_1(0)=1,\, q_0(0)=0,\, q_0'(0)=-1,\, q_1(0)=0$.  Actually, under some
restrictions, such systems are reducible to Li\'enard type differential
equation~\eqref{L2} with rational  functions $f$ and $g$, and in this
case the RCA described in Section~\ref{sec:RCA} is much more
efficient then the standard C-Algorithm.
\section{The C-Algorithm and the ReCA}
In~\cite{Chouikha2007}, the change of variable $u=\phi(x)=\xi + \int_{0}^{\xi} h(s)\,ds$ is introduced. Let us denote by $\tilde{g}(u)$ the function of $u$ represented by
both sides of~\eqref{CRI}
 \begin{equation}
\tilde g(u)=\frac {\xi(x)}{1+h(\xi(x))} = g(x) e^{F(x)}.
\end{equation} 
  The derivative of $u$ with respect to $x$ (resp. $\xi$) can easily be expressed  in
  terms of $x$ (resp. $\xi$),
\[\frac{du}{dx} = e^{F(x)} \mbox{ and } \frac{du}{d\xi} = 1 + h(\xi).\]

The C-Algorithm is based on the two sides computations of the
derivatives of $\tilde g(u)$ with respect to $u$ in terms of $\xi$ and
$x$, which consists in computing the following quantities:
\begin{eqnarray}\label{eq:Calgo}
\left.
  \begin{aligned}
    \tilde P_0(\xi) & =   \frac{\xi}{1+h(\xi)}\\
    \tilde P_k(\xi) &= \frac{d\tilde P_{k-1} (\xi)}{d\xi}\frac{1}{ (1+h(\xi))}  \\
    \tilde Q_0(x) &= g(x) e^{F(x)}\\
    \tilde Q_k(x) &=\frac{d \tilde Q_{k-1}(x)}{dx} e^{-F(x)}
  \end{aligned}
\right\}
\end{eqnarray}
where $k\geq 1$ and evaluating them at $0$, that is $\tilde P_k(0)=\tilde Q_k(0)$.
As $$\tilde P_k(0) = \frac{d^{k}\left(\frac{\xi}{1+h(\xi)}
  \right)}{du^k}\mid_{u=0}, \quad \mbox{and}\quad \tilde Q_k(0)
=\frac{d^{k}\left(g(x) e^{F(x)} \right)}{du^k}\mid_{u=0},$$ by
analycity,~\eqref{CRI} is equivalent to the equalities $\tilde
Q_k(0)=\tilde P_k(0)$, $k\geq0$.

In the next theorem we describe the Reduced C-Algorithm, which takes
care of significant multiplicative factors appearing in the formulas. The paper~\cite{BoussaadaChouikhaStrelcyn2010} is based on it.
\begin{thm}[Reduced C-Algorithm - ReCA]
\label{thmReCA}
  Let $f$ and $g$ be real analytic functions defined in a neighborhood of  \,$0$, such
  that $g(0)=0$, $g'(0)=1$ and $x g(x) > 0$ for $x \neq 0$. Then
  system~\eqref{L2} has an isochronous center at $O$ if and only if, there exists an odd analytic function $h$ such that
  for all $k\ge 0$,
\[\widehat P_k(0) = \widehat Q_k(0)\]
with the recursive formulas:
\begin{eqnarray}\label{eq:widehatPQ}
\left.
  \begin{aligned}
    \widehat P_0(\xi) & =  \xi\\
    \widehat P_k(\xi) &= \frac{d\widehat P_{k-1} (\xi)}{d\xi} (1+h(\xi)) - (2\,k-1)\,\widehat P_{k-1}(\xi) \frac{dh(\xi)}{d\xi} \\
    \widehat Q_0(x) &= g(x)\\
    \widehat Q_k(x) &=\frac{d \widehat Q_{k-1}(x)}{dx} - (k-2) f(x) \widehat
    Q_{k-1}(x)
  \end{aligned}
\right\}.
\end{eqnarray}
Moreover, for all $M>0$, the $M$ first necessary conditions of
isochronicity, $\widehat P_k(0)=\widehat Q_k(0)$ for $1\le k \le M$
can be obtained by computing the truncated power series expansions
around $0$ of $\widehat P_k(\xi)$ and $\widehat Q_k(x)$ up to the order
$M-k$.
\end{thm}
\begin{proof}
  First, the computation of the first two derivatives of
  $\tilde{g}(u)$ with respect to $u$ in terms of both $x$ and $\xi$ gives
\begin{eqnarray*}
  \tilde g'(u)&=&{\frac {1+h \left( \xi \right) -\xi\,{\frac {d h \left( \xi
 \right)}{d\xi}} }{ \left( 1+h \left( \xi \right)  \right) ^{3}}}={\frac {d}{d
x}}g \left( x \right) +g \left( x \right) f \left( x \right),
\end{eqnarray*}

\begin{eqnarray*}
  \tilde g''(u)&=&-{\frac {\left(1+h ( \xi )\right)\xi\,{\frac {d^{2} h \left( \xi \right)}{d{\xi}^{2}}}) +3\,{\frac {d h \left( \xi \right)}{d\xi}}\left( 1+h ( \xi )-\,\xi\,  {\frac {d h \left( \xi \right)}{d\xi}} 
 \right) }{ \left( 1+h \left( \xi \right)  \right) ^{5}}}\\
 &=& \left( 
{\frac {d^{2}g \left( x \right)}{d{x}^{2}}} + \left( {\frac {d g
 \left( x \right)}{dx}}  \right) f \left( x \right) +g \left( x \right) {
\frac {d f \left( x \right)}{dx}}  \right) {{\rm e}^{- F(x) }}.
\end{eqnarray*}

These formulas strongly suggest that the $k^{th}$ derivative of $\tilde{g}(u)$ can be written both in terms of $x$
and $\xi$ as follows:
\begin{eqnarray*}
  \tilde g^{(k)}(u)&=&\frac{\widehat P_k(\xi)}{(1+h(\xi))^{2k+1}}= \widehat Q_k(x)\,e^{(1-k)\,F(x)},
\end{eqnarray*}
with $\widehat P_k$ and $\widehat Q_k$ verifying the induction
formulas~\eqref{eq:widehatPQ}.

Then, we assume that this is so up to the order $n-1$. Differentiating
the two sides of the equality $\tilde g^{(n-1)}(u)$ with respect to
$u$ in terms of $x$ and $\xi$ gives
\begin{eqnarray*}
  \tilde g^{(n)}(u)&=&\frac{\widehat P_{n}(\xi)}{(1+h(\xi))^{2n+1}}= \widehat Q_n(x)\,e^{(1-n)\,F(x)},
\end{eqnarray*}
where $\widehat P_{n}(\xi)$ and $ \widehat Q_n(x)$ satisfy
formulas~\eqref{eq:widehatPQ} as expected.

As $F(0)=0$ and $h(0)=0$, necessary and sufficient conditions are
given by $\widehat P_k(0)=\widehat Q_k(0), k\ge 0$.

When our aim is to establish a fixed number $M$ of necessary
conditions, we restrict ourselves to the power series expansion around
$0$ of $\widehat P_k$ and $\widehat Q_k$ for $1\le k \le M$ up to the
order $M-k$, which is the minimal necessary truncation order. Indeed,
by formulas~\eqref{eq:widehatPQ}, to obtain $\widehat P_M(0)$ and
$\widehat Q_M(0)$ it is sufficient to compute a power series expansion
around $0$ of $\widehat P_M(\xi)$ and $\widehat Q_M(x)$ up to order
$0$ (i.e. constant terms) which require the power series expansion
around $0$ of $\widehat P_{M-1}(\xi)$ and $\widehat Q_{M-1}(x)$ up to
order $1$, and so on.
\end{proof}
In the practical use of the described algorithm we are concerned with
a finite number $M$ of necessary conditions. When a candidate for an
isochronous center is identified, we try to write down its Urabe
function under a closed-form expression, and prove the sufficiency
using again Theorem~\ref{thmC}.
\section{The RCA}
\label{sec:RCA}
In this section we restrict ourselves to systems~\eqref{L2P} for which
$f$ and $g$ are rational functions. For this particular case, we
describe an easy to handle couple of polynomial recursive formulas
which gives the $k$th derivatives of each side of~\eqref{CRI}.  Those
formulas apply in particular to systems~\eqref{CHERKAS} when they are
reducible to Li\'enard type differential
equation. The paper~\cite{BaBoChSt2010} is based on it.

We denote $f(x)=N_f(x)/D_f(x)$ (resp. $g(x)=N_g(x)/D_g(x)$), where
$N_f$, $N_g$, $D_f$ and $D_g$ are polynomials such that $D_f(0)=1$,
$D_g(0)=1$ and $\mbox{pgcd}(N_f,D_f)=1$, $\mbox{pgcd}(N_g,D_g)=1$.

\begin{thm}[Rational C-Algorithm - RCA] 
\label{thmRCA}
There exists a positive integer $M_0$, such that for any
  $M\ge M_0$ the following assertions are equivalents:
  \begin{enumerate}
  \item 
    the origin $O$ of system~\eqref{L2P} is an isochronous center;
  \item \label{item:2} there exists a real analytic odd function $h$ satisfying    \[ P_k(0)=Q_k(0)\] for all $0\le k \le M$,
    where
    \begin{equation}\label{eq:PQ}
      \left. \begin{aligned}P_0(\xi)&=\xi\\
          P_{k}(\xi)&=\left( {\frac {dP_{{k-1}} \left( \xi \right)}{d\xi}}  \right)  \left( 1+h \left( \xi \right)  \right) - \left( 2\,k-1 \right) P_{{k-1}} \left( \xi \right) {\frac {dh \left( \xi \right)}{d\xi}}\\
          Q_0(x)&=N_g(x)\\
          Q_k(x)&= { Q_{{k-1}} ( x )  D_g ( x ) \left(
                  \left( 1-k \right) {\frac {dD_f ( x )}{dx}} + \left(
                    2-k \right) N_f ( x ) \right) 
              }
          \\
          &{-k\, Q_{{k-1}} ( x )  {\frac
                    {dD_g ( x )}{dx}}  D_f ( x ) + {\frac {dQ_{{k-1}}\left( x \right)
                  }{dx}} D_g( x ) D_f ( x )}
        \end{aligned} \right\}.
    \end{equation}
  \end{enumerate}
  Moreover, as we only need the values of the $ P_k$ and $ Q_k$ at 0,
  it is sufficient to compute the power series expansions of $ P_k(x)$
  and $ Q_k(x)$ at order $M-k$, i.e. to truncate the polynomials $P_k$
  and $Q_k$ up to degree $M-k$.
\end{thm}
\begin{proof}
  As in the proof of the previous theorem, those formulas are found by
  induction on $k$:
\begin{equation}\label{CRA}
\tilde g^{(k)}(u)=\frac{ P_k(\xi)}{(1+h(\xi))^{2k+1}}
  =\frac{ Q_k(x)}{D_f(x)^k\,D_g(x)^{k+1}}e^{(1-k)F(x)}.
\end{equation} 

It remains to prove that there exists a finite $M_0$ such that the
$M_0$ first conditions are sufficient. This comes from the Hilbert
Basis Theorem, and more precisely the Ascending Chain Condition
(see~\cite{CoxLittleOShea2007}) applied to the ascending chain of
ideals $I_j = \langle P_k(0)-Q_k(0), \quad 0 \le k\le j \rangle$. Then
there exists an $M_0\ge 0$ such that $I_{M_0}=I_{M_0+1} = \dots = I_{\infty}$.
\end{proof}
\section{Efficiency of the RCA}
The original C-Algorithm which is based on~\eqref{eq:Calgo} will be
denoted by $A_0$.  In this section we study the efficiency of the
algorithms resulting from Theorems~\ref{thmReCA} and~\ref{thmRCA},
that will be denoted by $A_2,\,A_3,\,A_4$ and $A_5$:
\begin{itemize}
\item $A_1$ is a truncated C-algorithm. It is based on the
  formulas~\eqref{eq:Calgo}, for which we apply the truncation
  procedure using power series.
\item $A_2$ is the algorithm based on the
  formulas~\eqref{eq:widehatPQ} of Theorem~\ref{thmReCA} where the truncation
  procedure is applied, that is ReCA.
\item $A_3$ is the algorithm based on the
  formulas~\eqref{eq:widehatPQ} of Theorem~\ref{thmReCA} without truncations.
\item $A_4$ is the algorithm based on the formulas~\eqref{eq:PQ} of
  Theorem~\ref{thmRCA} where the truncation procedure is applied, that
  is RCA.
\item $A_5$ is the algorithm based on the formulas~\eqref{eq:PQ} of
  Theorem~\ref{thmRCA} without truncations.
\end{itemize}
To compare the efficiency of the above $6$ algorithms we will apply them to the quartic system 
\begin{equation}\label{L2PLANQUARTIC}
 \left. \begin{aligned}\dot x&= - y+ a_{1,1}xy+ a_{2,1}x^2y + a_{3,1}x^3y\\
 \dot y&= x + b_{2,0}x^2 + b_{3,0}x^3+ b_{0,2}y^2 + b_{1,2}xy^2 + b_{2,2}x^2y^2+ b_{4,0}x^4 \end{aligned}\right\}. 
\end{equation}
that is system $(15)$ from~\cite{BoussaadaChouikhaStrelcyn2010} or
system $(3.1)$ from~\cite{BaBoChSt2010}. By standard reduction this
system is reducible to Li\'enard type equation~\eqref{L2}
with $$f(x)={\frac
  {b_{{0,2}}+b_{{1,2}}x+b_{{2,2}}{x}^{2}+a_{{1,1}}+2\,a_{{2,1}}x+
    3\,a_{{3,1}}{x}^{2}}{1-a_{{1,1}}x-a_{{2,1}}{x}^{2}-a_{{3,1}}{x}^{3}}}
$$ and $$g(x)=\left( 1-a_{{1,1}}x-a_{{2,1}}{x}^{2}-a_{{3,1}}{x}^{3} \right) 
\left( x+b_{{2,0}}{x}^{2}+b_{{3,0}}{x}^{3}+b_{{4,0}}{x}^{4} \right)
.$$ There are 18 unknowns, 9 for the $a_{i,j}$ and $b_{i,j}$ and 9
for the coefficients of the power series expansion of $h$ up to order
17 (remember that $h$ is odd). Then, it is reasonable to compute the
conditions~\eqref{eq:PQ} at least up to order $M=19$. Since the depth of the isochronous center still an open problem for system \eqref{L2PLANQUARTIC}, then investigations need a higher number of necessary conditions. In our comparative study we ask for the first $30$ necessary conditions by each of the presented algorithms.

\begin{table}[ht]
\begin{center}
\begin{tabular}{|r|r|r|r|r|r|r|}
\hline\hline
Order of & C-Algorithm & & ReCA & & RCA & \\
  derivation  &  $A_0$ & $A_1$ & $A_2$ & $A_3$& $A_4$ & $A_5$ \\
  \hline
\hline
10& 230,8 &0,52 & 0,0 & 6,7&0,0& 5,7\\
\hline
15&5920,2  &19,1 &0,6  &523,4 &0,4&520,0\\
\hline
20&  &1168,6 &  4,6& &2,6&\\
\hline
 30&  & & 168,7 & &84,9&\\
\hline
\end{tabular}
\end{center}
  \caption{CPU time in seconds on Pentium 2,4 GHz with 4 Gb of memory
  }
\end{table}
The superiority of ReCA and RCA is obvious as well as the role of the
truncation.  The absence of values means that in that case the
computations failed by lack of memory.

\section{Examples and comments}
Let us recall that the isochronicity problem for planar cubic systems
(linear center perturbed by cubic nonlinearity) is still open. This
fact is due to the huge number of parameters (14 parameters).  In the
same time, several recent works have proven the power of the
algorithmic methods in the characterization of isochronous
centers. For instance we quote the normal forms approach established
by V. G. Romanovsky~\cite{RomanovskiShafer2009} and used in several of
his coauthored papers
\cite{ChenRomanovski2010,DolicaninMilovanovicRomanovski2007,ChenRomanovskiZhang2008}.
Particularly, this method has proven its performance in the study of
time-reversible isochronous centers. Indeed, the paper
\cite{ChenRomanovski2010} contains the complete set of time reversible
isochronous centers of linear center perturbed by cubic nonlinearity.
Hence the cubic isochronicity problem is still open only in the case
of non time-reversible systems.

Using Algorithm~RCA, we succeeded to establish several new cubic and
quartic isochronous
centers~\cite{BaBoChSt2010,BoussaadaChouikhaStrelcyn2010}.  Among
others, we found in~\cite{BaBoChSt2010} three families of new non
time-reversible cubic isochronous systems:
\begin{equation}\label{CUB1}
\left.\begin{aligned} \dot x &=-y-2\,b_{{2,0}}xy+{x}^{2}+2\,b_{{2,0}}{x}^{3}
\\ \dot y &=x-4\,b_{{2,0}}{y}^{2}-2\,xy+b_{{2,0}}{x}^{2}+4\,b_{{2,0}}{x}^{2}y+2\,{x}^{3}
 \end{aligned}\right\}, 
   \end{equation} 
 
\begin{equation}\label{CUB6} 
\left.\begin{aligned} \dot x &=-y\pm 2\,\sqrt
    {2}xy+{x}^{2}\mp 2\,\sqrt {2}{x}^{3}
\\ \dot y &=x\pm 8\,\sqrt {2}{y}^{2}-2\,xy\mp 3\,\sqrt {2}{x}
^{2}\mp 12\,\sqrt {2}{x}^{2}y+10\,{x}^{3}
 \end{aligned}\right\}, 
   \end{equation} 

\begin{equation}\label{CUB2}
\left.\begin{aligned} \dot x &=-y-\frac{1}{2}\,b_{{2,0}}xy+{x}^{2}+\frac{1}{2}\,b_{{2,0}}{x}^{3}
\\ \dot y &=x-b_{{2,0}}{y}^{2}-2\,xy+b_{{2,0}}{x}^{2}+b_{{2,0}}{x}^{2}y+\left( 2+\frac{1}{4}\,{b_{{2,0}}}^{2} \right) {x}^{3}
 \end{aligned}\right\}. 
   \end{equation}

   In \cite{ChenRomanovskiZhang2008} time-reversible isochronous
   centers of homogeneous quartic perturbation of the linear center
   are completely established.  In~\cite{BaBoChSt2010}, using RCA we
   found a large list of new non time-reversible quartic isochronous
   centers. We also found the following family of systems, for which
   we conjecture it has an isochronous center at~$0$:
   \begin{equation}\label{ST26} 
     \left. \begin{aligned}
         \dot x&=-y+ \left( -\frac{3}{8}-2\,b_{{2,2}} \right) {x}^{2}y+ \left( \frac{1}{16}+b_{{2,2}}
         \right) {x}^{3}y+xy
         \\
         \dot y&=x-\frac{3\,{x}^{2}}{4}+\frac{{y}^{2}}{4}+\frac{3\,{x}^{3}}{8}-2\,b_{{2,2}}x{y}^{2}+b_{{2,2}
         }{x}^{2}{y}^{2}-\frac{{x}^{4}}{16}
       \end{aligned}\right\}. 
   \end{equation}
   We were able to prove the isochronicity of the system in few
   particular cases, for instance for $b_{{2,2}}\in\{-
   \frac{1}{16},\,0,\, \frac{1}{16}\}$, by computing explicitly the
   Urabe function:
\begin{equation*}
h_{{\{b_{{2,2}}(\xi)=\frac{1}{16}\}}}={\frac {\sqrt {2}\xi\,\sqrt
    {2\,{\xi}^{2}+32} \left( {\xi}^{2}+12    \right) }{ 2\,\left(
      {\xi}^{2}+4 \right) \left( {\xi}^{2}+16    \right) }},
  \end{equation*} 
  \begin{equation*}   
h_{{\{b_{{2,2}}(\xi)=-\frac{1}{16}\}}}=\frac{\sqrt {2}\sqrt {2\,L
    \left(\frac{{\xi}^{2}}{4} \right) +    8}\sqrt {{\frac
      {{\xi}^{2}}{L \left( \frac{{\xi}^{2}}{4} \right)      }}} \left(
    L \left( \frac{{\xi}^{2}}{4}\right) +3 \right) L  \left(
    \frac{{\xi}^{2}}{4} \right)}{ 2\,{\xi} \left( L \left(
      \frac{{\xi}^{2}}{4} \right) +4 \right) \left( L \left(
      \frac{{\xi}^{2}}{4} \right) +1 \right)} ,
\end{equation*}
where  $L={\it  LambertW}$  is the Lambert function (see \cite{NIST2010}),
\begin{equation*}
h_{{\{b_{{2,2}}=0\}}}(\xi)=\frac{\sqrt {2}\sqrt {{\frac {-4+{\xi}^{2}+2\,\sqrt {4+2\,{\xi}^{2}}}{{\xi}^{2}}}}\xi\, \left( {\xi}^{2}+2\,\sqrt {4+2\,{\xi}^{2}}+2  \right)}{ \left( 2+{\xi}^{2} \right) \left( \sqrt {4+2\,{\xi}^{2}}+6  \right)} .
\end{equation*}


\paragraph{Acknowledgements}
We thank A. Raouf Chouikha (University Paris 13, France) 
for carefully
reading the manuscript of this note, as well as  Marie-Claude
Werquin (University Paris 13, France) who corrected our scientific
English. We are particularly indebted to Jean-Marie Strelcyn (University of Rouen, France) for his great help in final redaction of this note. Last but not least, we thank Jaume Llibre (Universitat Aut\`onoma de Barcelona) for valuable suggestions.

\bibliographystyle{plain}
\bibliography{BBCS}

\begin{thebibliography}{10}

\bibitem{BaBoChSt2010}
Magali Bardet, Islam Boussaada, Chouikha A.~Raouf, and Jean-Marie Strelcyn.
\newblock {Isochronicity conditions for some planar polynomial systems II}.
\newblock preprint, arXiv:1005.5048, 2010.

\bibitem{BoussaadaChouikhaStrelcyn2010}
Islam Boussaada, A.~Raouf Chouikha, and Jean-Marie Strelcyn.
\newblock Isochronicity conditions for some planar polynomial systems.
\newblock {\em Bulletin des Sciences Math\'ematiques}, In Press, 2010.

\bibitem{ChenRomanovskiZhang2008}
Xingwu Chen, Valery~G. Romanovksi, and Weinian Zhang.
\newblock Linearizability conditions of time-reversible quartic systems having
  homogeneous nonlinearities.
\newblock {\em Nonlinear Anal.}, 69(5-6):1525--1539, 2008.

\bibitem{ChenRomanovski2010}
Xingwu Chen and Valery~G. Romanovski.
\newblock {Linearizability conditions of time-reversible cubic systems.}
\newblock {\em J. Math. Anal. Appl.}, 362(2):438--449, 2010.

\bibitem{Chouikha2007}
A.Raouf Chouikha.
\newblock {Isochronous centers of Lienard type equations and applications.}
\newblock {\em J. Math. Anal. Appl.}, 331(1):358--376, 2007.

\bibitem{ChouikhaRomanovskiChen2007}
A.Raouf Chouikha, Valery~G. Romanovski, and Xingwu Chen.
\newblock {Isochronicity of analytic systems via Urabe's criterion.}
\newblock {\em J. Phys. A, Math. Theor.}, 40(10):2313--2327, 2007.

\bibitem{CoxLittleOShea2007}
David Cox, John Little, and Donal O'Shea.
\newblock {\em Ideals, varieties, and algorithms. An introduction to
  computational algebraic geometry and commutative algebra}.
\newblock Undergraduate Texts in Mathematics. Springer, New York, third
  edition, 2007.

\bibitem{DolicaninMilovanovicRomanovski2007}
Diana Dolicanin, Gradimir~V. Milovanovic, and Valery~G. Romanovski.
\newblock Linearizability conditions for a cubic system.
\newblock {\em Applied Mathematics and Computation}, 190(1):937 -- 945, 2007.

\bibitem{DumortierLlibreArtes2006}
Freddy Dumortier, Jaume Llibre, and Joan~C. Art{\'e}s.
\newblock {\em Qualitative theory of planar differential systems}.
\newblock Universitext. Springer-Verlag, Berlin, 2006.

\bibitem{NIST2010}
F.~W. J.~Olver et.~al. (Editors).
\newblock {\em NIST Handbook of Mathematical Functions}.
\newblock Cambridge Univ. Press, Cambridge, 2010.

\bibitem{RomanovskiShafer2009}
Valery~G. Romanovski and Douglas~S. Shafer.
\newblock {\em The center and cyclicity problems: a computational algebra
  approach}.
\newblock Birkh\"auser Boston Inc., Boston, MA, 2009.

\bibitem{Sabatini2004}
M.~Sabatini.
\newblock {On the period function of
  $x^{\prime\prime}+f(x)x^{\prime2}+g(x)=0$.}
\newblock {\em J. Differ. Equations}, 196(1):151--168, 2004.

\end{thebibliography}

\end{document}